\documentclass[12pt]{amsart}

\usepackage[margin=1in]{geometry}
\usepackage[T1]{fontenc}
\usepackage{newtxtext,newtxmath}
\usepackage{microtype}
\usepackage{latexsym,amscd, graphicx, color, amsthm, bm, amsmath, cancel, enumitem,mathtools}  
\usepackage{tikz, tikz-3dplot}
\usetikzlibrary{calc, backgrounds}
\usepackage{listofitems}
\usepackage{hyperref}
\usepackage{stackengine}

\numberwithin{equation}{section}

\newtheorem{theorem}{Theorem}[section]

\newtheorem{corollary}[theorem]{Corollary}
\newtheorem{lemma}[theorem]{Lemma}

\newtheorem{remark}[theorem]{Remark}
\newtheorem{definition}[theorem]{Definition}

\theoremstyle{definition}

\newcommand{\actseven}{%
  \mathord{\stackon[2pt]{7}{\scriptsize$\curvearrowright$}}%
}

\newcommand{\MRT}{{\mathrm{MRT}}}
\newcommand{\THC}{{\mathrm{THC}}}
\newcommand{\RHT}{{\mathrm{RHT}}}

\newcommand{\sort}{{\mathrm{sort}}}

\newcommand{\sign}{{\mathrm{sign}}}

\newcommand{\ch}{{\mathrm{ch}}}

\newcommand{\type}{{\mathrm{type}}}

\DeclareMathOperator{\Mob}{M\text{ö}b}
\newcommand{\symm}{{\mathfrak{S}}}

\newcommand{\SSS}{{\mathcal{S}}}

\newcommand{\PPPP}{{\mathfrak{P}}}

\newcommand{\CC}{{\mathbb{C}}}

\newcommand{\ZZ}{{\mathbb{Z}}}

\newcommand{\MMM}{\mathcal{M}}

\newcommand{\cyc}{\text{cyc}}

\begin{document}

\title[Forgotten characters]
{Forgotten characters}

\author{Kyle Celano}

\address{Department of Mathematics, Wake Forest University, Winston-Salem, NC, 27109, USA}
\email{celanok@wfu.edu}

\author{Brendon Rhoades}

\address{Department of Mathematics, UC San Diego, La Jolla, CA, 92039, USA}
\email{bprhoades@ucsd.edu}

\begin{abstract}
    A {\em partial permutation} of $[n] := \{1,\dots,n\}$ is a bijection $g: I \to J$ between two subsets $I,J \subseteq [n]$. Given a partial permutation $g$ of $[n]$, let $a_g \in \CC[\symm_n]$ be the group algebra sum of those permutations $w \in \symm_n$ which extend $g$. Informally, a partial permutation $g$ is obtained by forgetting some data in a genuine permutation.
    The forgotten symmetric functions  are the least-studied of the six `standard' bases for the ring of symmetric functions. We show that forgotten symmetric functions arise naturally in class function evaluations on partial permutations.
\end{abstract}

\maketitle

\section{Introduction and Main Result}
\label{sec:Introduction}

Let $\symm_n$ be the symmetric group on $[n] : = \{1,\dots,n\}$ and let $\CC[\symm_n]$ be its group algebra. Given a class function $\chi: \symm_n \to \CC$, we have the linear extension $\chi: \CC[\symm_n] \to \CC$ given by
\begin{equation}
\label{eq:main-character-problem}
   \chi(a) =   \chi\left( \sum_{w \in \symm_n} c_w \cdot w \right) = \sum_{w \in \symm_n} c_w \cdot \chi(w) \quad \quad \text{for } 
   a = \sum_{w \in \symm_n} c_w \cdot w \in \CC[\symm_n].
\end{equation}
 An important (and usually intractable) problem in algebraic combinatorics seeks to obtain meaningful, useful, or efficient interpretations of $\chi(a)$ when $\chi$ is an important class function and/or $a$ is an important group algebra element. 


We study group algebra elements obtained by forgetting some of the data in a permutation. Let $X$ be a finite set. A {\em partial permutation} of $X$ is a bijection $g: I \to J$ between two subsets $I,J \subseteq X$. The ambient set $X$ is part of the data of a partial permutation, but we avoid the more cumbersome notation $(g,X)$.  We write 
\begin{equation}
    \PPPP_X(I,J) := \{ \text{all partial permutations $g:I\to J$}\} 
\end{equation}
for the set of all partial permutations $g: I \to J$ with ambient set $X$ and
\begin{equation}
    \PPPP_X := \bigsqcup_{\substack{I,J \subseteq X \\ |I| = |J|}} \PPPP_X(I,J)
\end{equation}
for the set of all partial permutations of $X$. When $X = [n]$, we write $\PPPP_n(I,J)$ and $\PPPP_n$ instead of $\PPPP_{[n]}(I,J)$ and $\PPPP_{[n]}$.
For example, if $n = 8,\ I = \{1,2,4,5,7\},$ and $J = \{2,3,4,5,7\}$,  a possible partial permutation $g \in \PPPP_8(I,J)$ is
\[
g(1) = 5, \quad g(2) = 4, \quad g(4) = 2, \quad g(5) = 3, \quad g(7) = 7.
\]
Partial permutations in $\PPPP_n$ may be regarded as 0,1-matrices of size $n$ with at most one 1 in each row and column; this yields the {\em rook monoid}  \cite{Solomon02}.

Let  $g \in \PPPP_n(I,J)$ be a partial permutation. We define $a_g \in \CC[\symm_n]$ to be the group algebra element
\begin{equation}
    a_g : =\sum_{\substack{w \in \symm_n \\ w(i) = g(i) \text{ for all } i \in I}} w.
\end{equation}
In the example above, $a_g \in \CC[\symm_8]$ is a sum of $(8-5)! = 6$ permutations in $\symm_8$. Partial permutations have appeared in the context of permutation statistic asymptotics \cite{hamaker2025moments, loth2023permutation}. 

A partial permutation $g \in \PPPP_n(I,J)$  is encoded by a directed graph on the vertex set $[n]$ with edges $i \to g(i)$ for all $i \in I$. Each component of this graph is a directed path or a directed cycle. The {\em path type} $\mu = (\mu_1, \dots, \mu_k)$ of $g$ is obtained by listing the path sizes of this graph in weakly decreasing order\footnote{Here the size of a path is the number of {\bf vertices} in that path.}. The {\em cycle type} $\nu = (\nu_1,\dots,\nu_r)$ of $g$ is similarly obtained by listing the cycle sizes in $g$ in weakly decreasing order. Thus $\mu$ and $\nu$ are partitions with $|\mu| + |\nu| = n$. In the example from the last paragraph, we have the directed graph
\[
1 \to 5 \to 3 \quad \quad 2 \leftrightarrow 4 \quad \quad 6 \quad \quad \actseven \quad \quad  8
\]
with path type $\mu = (3,1,1)$ and cycle type $\nu = (2,1)$.

Let $\Lambda = \bigoplus_{\lambda \vdash n} \Lambda_n$ be the graded ring of symmetric functions over the ground field $\CC$. We assume familiarity with basic symmetric function theory as can be found in e.g. \cite{macdonald1998symmetric,stanley1999enumerative}. Bases of $\Lambda_n$ are indexed by partitions $\lambda \vdash n$. The `standard' six bases are the {\em monomial basis} $\{m_\lambda\}$, the {\em elementary basis} $\{ e_\lambda \}$, the {\em complete homogeneous basis} $\{ h_\lambda \}$, the {\em power sum basis} $\{p_\lambda\}$, the {\em forgotten basis} $\{f_\lambda\}$, and the {\em Schur basis} $\{s_\lambda\}$. Let $\langle -,-\rangle: \Lambda \times \Lambda \to \CC$ be the {\em Hall inner product} with respect to which the Schur basis $\{s_\lambda\}$ is orthonormal. Let $\omega: \Lambda \to \Lambda$ be the  involution on $\Lambda$ characterized by $\omega(e_\lambda) = h_\lambda$. Writing 
\begin{equation}
    R_n := \{ \text{all class functions $\chi: \symm_n \to \CC$} \}
\end{equation}
for the vector space of class functions $\chi: \symm_n \to \CC$, we have the {\em Frobenius characteristic isomorphism} $\ch_n: R_n \to \Lambda_n$ with formula
\begin{equation}
    \ch_n(\chi) :=  \frac{1}{n!} \sum_{w \in \symm_n} \chi(w) \cdot p_{\cyc(w)},
\end{equation}
where $\cyc(w) \vdash n$ is the cycle type of $w$. Let $R := \bigoplus_{n \geq 0} R_n$.
We extend $\ch_n$ to a graded vector space isomorphism 
\begin{equation}
\ch: R \to \Lambda \quad \text{where} \quad \ch := \bigoplus_{n \geq 0} \ch_n.
\end{equation}

As the nomenclature suggests, the forgotten symmetric functions $f_\lambda$ are the least studied of the standard six bases of $\Lambda$. 
They are related to the monomial symmetric functions by 
\begin{equation}
    f_\lambda = \omega(m_\lambda).
\end{equation}
We show that the $f_\lambda$ arise naturally in class function evaluations on partial permutation group algebra elements $a_g$. Given $\mu \vdash n$, let $\ell(\mu)$ be the number of parts of $\mu$ and write
\[
m(\mu)! := \prod_{i \geq 1} m_i(\mu)!,
\]
where $m_i(\mu)$ is the multiplicity of $i$ as a part of $\mu$.

\begin{theorem}
    \label{thm:forgotten}
    Let $g \in \PPPP_n$ have path type $\mu$ and cycle type $\nu$ and let $\chi: \symm_n \to\CC$ be a class function. We have
    \begin{equation}
        \chi(a_g) = (-1)^{|\mu| - \ell(\mu)} m(\mu)! \cdot  \langle \ch(\chi), f_\mu \cdot p_\nu \rangle.
    \end{equation}
\end{theorem}

Theorem~\ref{thm:forgotten} says that class function evaluations on the group algebra elements $a_g$ (obtained by {\em forgetting} some of the data of a permutation in $\symm_n$) are Hall inner product evaluations involving {\em forgotten} symmetric functions. The forgotten basis $\{f_\lambda\}$ was therefore well-named.
For the partial permutation in our running example, we have $\mu = (3,1,1)$ and $\nu = (2,1)$ so that 
\[
m(\mu)! = m_1(\mu)! \cdot m_2(\mu)! \cdot m_3(\mu)! = 2! \cdot 0! \cdot 1! = 2
\]
and
\[
\chi(a_g) = (-1)^{5-3} \cdot 2 \cdot \langle \ch(\chi), f_{311} \cdot p_{21} \rangle 
\]
for any class function $\chi: \symm_8 \to \CC$.  In Theorem~\ref{thm:path-power-to-schur} we give a tableau rule for evaluating Theorem~\ref{thm:forgotten} when $\chi = \chi^\lambda$ is an irreducible character. Given two subsets $I,J \subseteq [n]$ of the same size, we have the group algebra element
\[
\{ I,J\} := \sum_{\substack{w \in \symm_n \\ w(I) = J}} w \in \CC[\symm_n]
\]
obtained by only `remembering' that a permutation $w$ carries the set $I$ to the set $J$, but forgetting all else. In Theorem~\ref{thm:set-irreducible-character} we apply our combinatorial rule to give an explicit formula for the irreducible character evaluations $\chi^\lambda\{I,J\}.$ See \cite{Lindzey, Romero} for further contemporary applications of the forgotten basis.

The rest of the paper is organized as follows. In {\bf Section~\ref{sec:Proof}}, we prove Theorem~\ref{thm:forgotten} using a result of Doubilet and the M\"obius function of the set partition lattice. {\bf Section~\ref{sec:Monotonic}} describes a combinatorial rule for  Theorem~\ref{thm:forgotten} when $\chi = \chi^\lambda$ is an irreducible character of $\symm_n$. {\bf Section~\ref{sec:Elements}} applies this combinatorial rule to compute the value of $\chi^\lambda\{I,J\}.$

\section{Proof of Theorem~\ref{thm:forgotten}}
\label{sec:Proof}

Our proof of Theorem~\ref{thm:forgotten} rests on a nearly-forgotten result from the Ph.D. thesis of P. Doubilet. In order to state Doubilet's theorem, we need some notation. We refer the reader to e.g. \cite{stanley2012enumerative} for the relevant background.

Let $\Pi_n$ be the lattice of set partitions of $[n]$ with the order
\[
\sigma \leq \pi \quad \Leftrightarrow \quad \sigma \text{ refines } \pi.
\]
The lattice $\Pi_n$ has minimum element $\hat{0} = \{1 / 2 / \cdots / n \}$ and maximum element $\hat{1} = \{1,2,\dots,n\}$. Given $\sigma \leq \pi$ in $\Pi_n$, each block $B$ of $\pi$ is a disjoint union of blocks $C^B_1, C^B_2, \dots, C^B_{s_B}$ of $\sigma$ for some $s_B > 0$.
The interval $[\sigma,\pi]$ is isomorphic to the direct product
\begin{equation}
\label{eq:partition-lattice-intervals}
    [\sigma,\pi] \cong \prod_{B \in \pi} \Pi_{s_B}
\end{equation}
of smaller partition lattices.

Let $\Mob_{\Pi_n}: \Pi_n \times \Pi_n \to \ZZ$ be the M\"obius function of $\Pi_n$. It is well-known that
\begin{equation}
    \label{eq:full-lattice}
   \Mob_{\Pi_n} (\hat{0},\hat{1}) = (-1)^{n-1} (n-1)!.
\end{equation}
Applying \eqref{eq:partition-lattice-intervals} and the multiplicative property of M\"obius functions
\begin{equation}
\label{eq:mobius-product}
    \Mob_{P \times Q}((p,q),(p',q')) = \Mob_P(p,p') \cdot \Mob_Q(q,q')
\end{equation}
on arbitrary finite posets $P,Q$ gives a formula for $\Mob_{\Pi_n}[\sigma,\pi]$ for all $\sigma \leq \pi$. In particular, we have
\begin{equation}
    \label{eq:general-mobius}
    \Mob(\sigma,\pi) = \prod_{B \in \pi} (-1)^{s_B-1} \cdot (s_B - 1)! \quad 
    \text{where $B \in \pi$ is a union of $s_B$ blocks of $\sigma$.}
\end{equation}

Given a set partition $\pi$ of $[n]$, the {\em type} of $\pi$ is the integer partition $\type(\pi) \vdash n$ obtained by listing the block sizes of $\pi$ in weakly decreasing order. For example, if $\pi = \{1, 3,5 / 2,6/4,7\}$ we have $\type(\pi) = (3,2,2)$. Doubilet's result is as follows.

\begin{theorem} {\em (Doubilet {\cite[Theorem 8(iii)]{Doubilet1972}}) }
    \label{thm:doubilet}
    Let $\mu \vdash n$ and let $\sigma$ be a set partition of $[n]$ with $\type(\sigma) = \mu$. We have 
    \begin{equation}
        f_\mu = \frac{(-1)^{n - \ell(\mu)}}{m(\mu)!} \cdot  \sum_{\sigma \leq \pi} |\textnormal{M\"{o}b}(\sigma,\pi)|  \cdot p_{\type(\pi)}. 
    \end{equation}
\end{theorem}

\begin{remark}
    \label{rmk:notation-warning}
    Doubilet's notation differs significantly from modern conventions. He writes $k_\lambda$ for the monomial symmetric function and $s_\lambda$ for the power sum symmetric function ({\bf not} the Schur function). He uses $\theta: \Lambda \to \Lambda$ to denote the $\omega$ involution. We use the now-standard definition of the forgotten basis as the image under the monomial basis of this involution, but Doubilet's $f_\mu$ differs from ours by a factor of $(-1)^{n - \ell(\mu)}$.
\end{remark}

We use Theorem~\ref{thm:doubilet} to prove Theorem~\ref{thm:forgotten}. Let $g \in \PPPP_n(I,J)$ be a partial permutation with path type $\mu$ and cycle type $\nu$ where $I,J \subseteq [n]$ satisfy $|I|=|J|$. Let $\chi: \symm_n \to \CC$ be a class function. For any $w \in \symm_n$, we have
\begin{equation}
    \chi(w) = \langle \ch(\chi), p_{\cyc(w)} \rangle,
\end{equation}
where $\cyc(w) \vdash n$ is the cycle type of $w$. 
The directed graph of $g$ partitions its vertex set $[n]$ into disjoint cycles and paths. Let $\sigma$ be the set partition of 
\[
P_g: =\{ 1 \leq i \leq n \,:\, \text{$i$ lies in a path of $g$} \}
\]
whose blocks are the paths. For our example partial permutation of $[8]$ with graph 
\[
1 \to 5 \to 3 \quad \quad 2 \leftrightarrow 4 \quad \quad 6 \quad \quad \actseven \quad \quad  8,
\]
we have 
\[P_g = \{1,3,5,6,8\} \quad \text{ and  } \quad  \sigma = \{ 1, 3,5 / 6 / 8\}.\] Observe that $\type(\sigma) = \mu$.
Theorem~\ref{thm:forgotten} will be proven if we can show
\begin{equation}
    \label{eq:combinatorial-expansion}
    \sum_{\substack{w \in \symm_n \\ w(i) = g(i) \text{ for all } i \in I}} p_{\cyc(w)} =  p_\nu \cdot \sum_{\sigma \leq \pi} |\Mob(\sigma,\pi)| \cdot p_{\type(\pi)}.
\end{equation}

Equation~\eqref{eq:combinatorial-expansion} may be verified combinatorially as follows. A typical permutation $w \in \symm_n$ for which $w(i) = g(i)$ for all $i \in I$ is obtained by concatenating
the directed paths of $g$ to form directed cycles until no directed paths remain.  One way to do this for our example partial permutation yields the cycle notation
\[
w = (\underline{1,5,3} \, , \, \underline{8} ) (2,4) (\underline{6}) (7) \in \symm_8
\]
where the directed paths in the partial permutation $g$ have been underlined.
Any permutation $w \in \symm_n$ so obtained yields a set partition $\pi(w)$ of $P_g$ where $i \sim j$ if and only if $i$ and $j$ belong to the same cycle of $w$. In our example we have
\[
\pi(w) = \{ 1,3,5,8 / 6 \}.
\]
We have the refinement relation $\sigma \leq \pi(w)$, and the permutation $w \in \symm_n$ contributes 
\begin{equation}
    p_\nu \cdot p_{\type(\pi(w))} 
\end{equation}
to the left-hand side of \eqref{eq:combinatorial-expansion}.

Fix a set partition $\pi$ of $P_g$ with $\sigma \leq \pi$. Each block $B$ of $\pi$ is a disjoint union of blocks $C^B_1, C^B_2, \dots, C^B_{s_B}$ of $\sigma$, each of which is a path in $g$. The summand of the right-hand side of \eqref{eq:combinatorial-expansion} indexed by $\pi$  contributes 
\begin{equation}
\label{eq:rhs-contribution}
    p_\nu  \cdot |\Mob(\sigma,\pi)| \cdot p_{\type(\pi)} = \prod_{B \in \pi} (s_B - 1)! \cdot p_\nu \cdot p_{\type(\pi)},
\end{equation}
where we applied Equation~\eqref{eq:general-mobius}. On the other hand, a typical permutation $w \in \symm_n$ indexing the left-hand side of \eqref{eq:combinatorial-expansion} for which $\pi(w) = \pi$ is obtained by placing a cyclic order on the paths $C^B_1, C^B_2, \dots, C^B_{s_B}$ for each block $B \in \pi$; there are $\prod_{B \in \pi} (s_B-1)!$ ways to do this. We conclude that summing \eqref{eq:rhs-contribution} over all set partitions $\pi$ with $\sigma \leq \pi$ gives the left-hand side of \eqref{eq:combinatorial-expansion}. This proves \eqref{eq:combinatorial-expansion} and completes the proof of Theorem~\ref{thm:forgotten}.

\section{Monotonic Ribbon Tilings}
\label{sec:Monotonic}

Hamaker and Rhoades \cite{HR} gave a combinatorial interpretation of the character evaluation $\chi(a_g)$ in Theorem~\ref{thm:forgotten} when $\chi = \chi^\lambda$ is the irreducible character of $\symm_n$ associated to $\lambda \vdash n$. We use a theorem of Allen and Mason to give a shorter proof of this result. 
The combinatorial objects introduced by Hamaker and Rhoades are as follows.

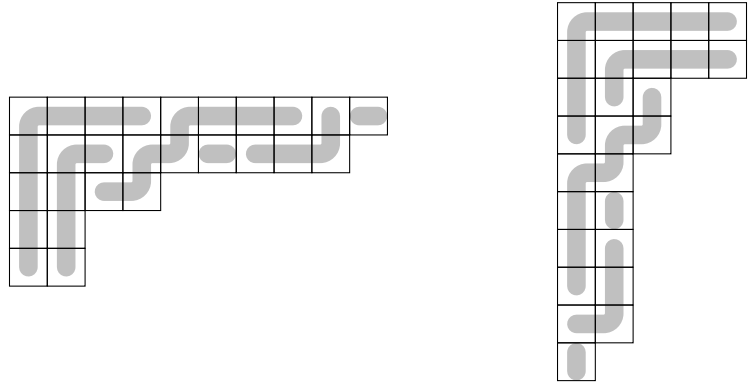
\begin{figure}
\begin{center}
\begin{tikzpicture}
     \node (A) at (0,0) {\begin{tikzpicture}[yscale=-1,scale=.5]

\def\L{{10,9,4,2,2}}
\pgfmathsetmacro{\len}{dim(\L)}

\foreach \y in {1,...,\len}{
    \pgfmathsetmacro{\j}{\L[\y-1]}
    \foreach \i in {1,...,\j}{
        \draw (\i-.5,\y-.5) rectangle (\i+.5,\y+.5);
        }
}

\begin{scope}[on background layer]
\tikzset{every path/.style={line width = 7pt,color=black,line cap=round,opacity=.25,rounded corners}}
\draw (10+.25,1)--(10-.25,1);
\draw (9,1)--(9,2)--(7,2);
\draw (6+.25,2)--(6-.25,2);
\draw (8,1)--(5,1)--(5,2)--(4,2)--(4,3)--(3,3);
\draw (4,1)--(1,1)--(1,5);
\draw (3,2)--(2,2)--(2,5);
\end{scope}
\end{tikzpicture}};
\node (B) at (6,0) {
\begin{tikzpicture}[yscale=-1,rotate=-45, xscale=-1,rotate=45, scale=.5]
\def\L{{10,9,4,2,2}}
\pgfmathsetmacro{\len}{dim(\L)}

\foreach \y in {1,...,\len}{
    \pgfmathsetmacro{\j}{\L[\y-1]}
    \foreach \i in {1,...,\j}{
        \draw (\i-.5,\y-.5) rectangle (\i+.5,\y+.5);
        }
}

\begin{scope}[on background layer]
\tikzset{every path/.style={line width = 7pt,color=black,line cap=round,opacity=.25,rounded corners}}
\draw (10+.25,1)--(10-.25,1);
\draw (9,1)--(9,2)--(7,2);
\draw (6+.25,2)--(6-.25,2);
\draw (8,1)--(5,1)--(5,2)--(4,2)--(4,3)--(3,3);
\draw (4,1)--(1,1)--(1,5);
\draw (3,2)--(2,2)--(2,5);
\end{scope}
\begin{scope}[shift={(.5,.5)}]



    
    
\end{scope}
\end{tikzpicture}

};

\end{tikzpicture}
\end{center}
\caption{A monotonic ribbon tiling (left) and a tunnel hook covering (right).}
\label{fig:t-m}
\end{figure}


\begin{definition}
\label{def:mrt} {\em (Hamaker--Rhoades \cite{HR})}
    A {\em monotonic ribbon tiling} of a partition shape $\lambda$ is a covering of $\lambda$ by lattice paths (called {\em ribbons}) that start on the lower boundary of $\lambda$ and progress up and to the right. The {\em sign} of a monotonic ribbon tiling $T$ is $\sign(T):=(-1)^k$, where $k$ is the total number of up steps taken by the ribbons. The {\em type} of $T$ is the composition $\mu=(\mu_1,\dots,\mu_r)$ where $\mu_i$ is the size of the ribbon starting in the $i$-th leftmost column containing the start of a ribbon. Let
    \[
        \MRT_{\lambda,\mu} := \{ \text{all monotonic ribbon tilings of shape $\lambda$ and type $\mu$} \}.
    \]
\end{definition}

An example monotonic ribbon tiling $T$ is shown on the left of Figure~\ref{fig:t-m}. The shape of $T$ is $\lambda = (10,9,4,2,2)$, the type is $\mu = (8,5,8,1,4,1)$, and the sign is $(-1)^{4+3+2+0+1+0} = +1$.
Allen and Mason introduced a similar set of tilings in their study of the inverse Kostka matrix.\footnote{In fact, Allen--Mason defined a more general set of tilings which cover composition diagrams. Definition~\ref{def:thc} is their construction in the special case of partitions.}

\begin{definition}
\label{def:thc}
{\em (Allen--Mason \cite{allen_combinatorial_2025})}
    A {\em tunnel hook covering} of a partition shape $\lambda$ is a covering of $\lambda$ by lattice paths (called tunnel hooks) that start on the right boundary of $\lambda$ and progress down and to the left. The {\em sign} of a tunnel hook covering $T$ is $\sign(T):=(-1)^k$, where $k$ is the total number of down steps taken by the tunnel hooks. The {\em type} of $T$ is the composition $\mu=(\mu_1,\dots,\mu_r)$ where $\mu_i$ is the size of the tunnel hook starting in the $i$-th topmost row containing the start of a tunnel hook.  
    \[
        \THC_{\lambda,\mu} := \{\text{all tunnel hook coverings of shape $\lambda$ and type $\mu$} \}.
    \]
\end{definition}

A tunnel hook covering $T$ of shape $\lambda = (5,5,3,3,2,2,2,2,2,1)$ is shown on the right of Figure~\ref{fig:t-m}. The type of $T$ is $\mu = (8,5,8,1,4,1)$ and the sign is $(-1)^{3+1+5+0+2+0} = +1.$ As suggested by Figure~\ref{fig:t-m}, monotonic ribbon tilings and tunnel hook coverings are essentially the same objects. For a partition $\lambda \vdash n$, write $\lambda'$ for the conjugate partition obtained by transposing across the main diagonal.

\begin{lemma}\label{lem:conjugation_bijection}
    Let $\lambda$ be a partition of $n$ and $\mu$ be a composition of $n$. Define $\tau:\MRT_{\lambda,\mu}\to \THC_{\lambda',\mu}$ by letting $\tau(T)$ be the transpose of a monotonic ribbon tiling $T$. Then $\tau$ is a bijection and we have
    \[\sign(T) \cdot \sign(\tau(T))=(-1)^{n-\ell(\mu)}\]
    for all $T \in \MRT_{\lambda,\mu}$.
\end{lemma}
\begin{proof}
    It is immediate from Definitions~\ref{def:mrt} and \ref{def:thc} that $\tau$ is a bijection. Since transposing a lattice path interchanges vertical and horizontal steps, the effect on signs is as described.
\end{proof}

With Lemma~\ref{lem:conjugation_bijection} in hand, the Schur expansion of $f_\mu$ follows quickly from the work of Allen and Mason \cite{allen_combinatorial_2025}.  Given a (strong) composition $\alpha$ of $n$, let $\sort(\alpha) \vdash n$ be the partition obtained by writing the parts of $\alpha$ in weakly decreasing order.


\begin{theorem}
    \label{thm:path-power-to-schur}
    For any partition $\mu \vdash n$ we have 
    \[
    (-1)^{n - \ell(\mu)}  \cdot f_\mu = \sum_{\lambda \vdash n} \left[ 
    \sum_{\sort(\alpha) = \mu}
    \sum_{T \in \MRT_{\lambda,\alpha}} \sign(T) \right] \cdot s_\lambda,
    \]
    where the second sum is over all strong compositions $\alpha$ with $\sort(\alpha) = \mu$.
\end{theorem}

\begin{proof}
    The {\em Kostka matrix} $K = (K_{\lambda,\mu})$ is the transition matrix from the Schur basis to the monomial basis, i.e.
    \begin{equation}
        s_\lambda = \sum_{\mu \vdash n} K_{\lambda,\mu} \cdot m_\mu \quad \quad \text{for all } \lambda \vdash n.
    \end{equation}
    The inverse Kostka matrix $K^{-1} = (K^{-1}_{\mu,\lambda})$ therefore satisfies
    \begin{equation}
        m_\mu = \sum_{\lambda \vdash n} K^{-1}_{\mu,\lambda} \cdot s_\lambda \quad \quad \text{for all $\mu \vdash n$.}
    \end{equation}
    Since $\omega(m_\mu) = f_\mu$ and $\omega(s_\lambda) = s_{\lambda'}$, applying $\omega$ to both sides of this equation gives
    \begin{equation}
    \label{eq:f-inverse-kostka}
        f_\mu = \sum_{\lambda \vdash n} K^{-1}_{\mu,\lambda'} \cdot s_\lambda.
    \end{equation}
    Allen and Mason \cite{allen_combinatorial_2025} gave the following combinatorial interpretation of the inverse Kostka matrix entry $K^{-1}_{\mu,\lambda'}$:
    \begin{equation}
    \label{eq:am-thc}
        K^{-1}_{\mu,\lambda'} = \sum_{\sort(\alpha) = \mu} \sum_{T \in \THC_{\lambda',\alpha}} \sign(T).
    \end{equation}
    Equation~\eqref{eq:am-thc} may also be derived from earlier work of E\u{g}ecio\u{g}lu and Remmel \cite{egecioglu_combinatorial_1990}.
    Combining Equations~\eqref{eq:f-inverse-kostka} and \eqref{eq:am-thc} with Lemma~\ref{lem:conjugation_bijection} completes the proof.
\end{proof}

For a partition $\mu \vdash n$, Hamaker and Rhoades defined \cite[Def. 4.6]{HR} the ``path power sum''  $\vec{p}_\mu$ via the Frobenius isomorphism $\ch$.  In particular, let $g \in \PPPP_n(I,J)$ have path type $\mu$ and empty cycle type. The partial permutation $g$ gives rise to a map ${\bf 1}_g: \symm_n \to \CC$ given by
\begin{equation}
    {\bf 1}_g(w) = \begin{cases}
        1 & \text{if $w(i) = g(i)$ for all $i \in I,$}\\
        0 & \text{otherwise.}
    \end{cases}
\end{equation}
Conjugation gives a class function $\psi_\mu: \symm_n \to \CC$ with formula
\begin{equation}
  \psi_\mu(w) := \sum_{u \in \symm_n} {\bf 1}_g(uwu^{-1}).  
\end{equation}
It is not hard to see that  $\psi_\mu$ depends only on $\mu$ and is otherwise independent of $g$. The {\em path power sum} symmetric function as defined in \cite{HR} is the image
\begin{equation}
    \vec{p}_\mu := \ch(\psi_\mu) \in \Lambda_n
\end{equation}
of $\psi_\mu$ under the Frobenius isomorphism. Combining \cite[Thm. 5.13]{HR} and Theorem~\ref{thm:path-power-to-schur} shows
\begin{equation}
\label{eq:pp-to-f}
    \vec{p}_\mu = (-1)^{n - \ell(\mu)} m(\mu)! \cdot f_\mu
\end{equation}
so that the path power sum as defined in \cite{HR} is a rescaled forgotten symmetric function.  
Equation~\eqref{eq:pp-to-f} can also be derived by comparing Doubilet's Theorem~\ref{thm:doubilet} with \cite[Prop. 4.9]{HR}.

Let $g \in \PPPP_n$ have path type $\mu$ and cycle type $\nu$. Combining \eqref{eq:pp-to-f} with Theorem~\ref{thm:forgotten} gives
\begin{equation}
\label{eq:cycle-path}
    \chi(a_g) = \langle \ch(\chi), \vec{p}_\mu \cdot p_\nu \rangle \quad \quad \text{for any class function } \chi: \symm_n\to \CC.
\end{equation}
Equation~\eqref{eq:cycle-path} is equivalent to \cite[Prop. 4.7]{HR}.

Given a skew shape $\lambda/\mu$, a {\em rim hook tableau} $T$ of skew shape $\lambda/\mu$ is a sequence
\[
T = ( \mu = \lambda^0 \subseteq \lambda^1 \subseteq \cdots \subseteq \lambda^m  = \lambda)
\]
such that $\xi^i := \lambda^i / \lambda^{i-1}$ is a ribbon for all $i =1,\dots,m$. The {\em sign} of $T$ is 
\begin{equation}
    \sign(T) := \prod_{i=1}^m \sign(\xi^i),
\end{equation}
where $\sign(\xi^i)$ is $(-1)^{r-1}$ when $\xi^i$ occupies $r$ rows. The {\em type} of $T$ is the sequence $( |\xi^1|, |\xi^2|, \dots , |\xi^m|)$ of ribbon sizes. The {\em Murnaghan--Nakayama rule} states  
\begin{equation}
    \label{eq:mn-rule}
    s_\mu \cdot p_\nu = \sum_\lambda  \left[  \sum_{T \in \RHT_{\lambda/\mu,\nu}} \sign(T) \right] \cdot s_\lambda,
\end{equation}
where the inner sum is over the set 
\begin{equation}
    \RHT_{\lambda/\mu,\nu} := \{ \text{all rim hook tableaux of shape $\lambda/\mu$ and type $\nu$} \}.
\end{equation}
Theorem~\ref{thm:forgotten} and Theorem~\ref{thm:path-power-to-schur} combine with the Murnaghan--Nakayama rule to give a combinatorial expression for $\chi^\lambda(a_g)$ for any $\lambda \vdash n$ and any $g \in \PPPP_n$.


\begin{corollary}
    \label{cor:partial-permutation-character}
    Let $g \in \PPPP_n$ have path type $\mu$ and cycle type $\nu$.  For $\lambda \vdash n$ we have
    \begin{align}
        \chi^\lambda(a_g) &= (-1)^{|\mu| - \ell(\mu)} \cdot m(\mu)! \cdot \langle s_\lambda, f_\mu \cdot p_\nu \rangle \\
        &= \langle s_\lambda, \vec{p}_\mu \cdot p_\nu \rangle \\
        &= 
        m(\mu)! \cdot \sum_{\rho \subseteq \lambda} \left[  \sum_{\substack{\sort(\alpha) = \mu \\ T_1 \in \MRT_{\rho,\alpha}}} \sign(T_1) \right] \cdot \left[ \sum_{T_2 \in \RHT_{\lambda/\rho,\nu}} \sign(T_2) \right].
    \end{align}
\end{corollary}

\begin{proof}
    The first equality is a restatement of Theorem~\ref{thm:forgotten} and the second equality follows immediately from Equation~\eqref{eq:pp-to-f}. For the third equality, recall that the Schur basis is orthonormal and apply Theorem~\ref{thm:path-power-to-schur} together with the Murnaghan--Nakayama rule \eqref{eq:mn-rule}.
\end{proof}

\section{The group algebra elements $\{I,J\}$}
\label{sec:Elements}

Let $a \in \CC[\symm_n]$ be a group algebra element and suppose we have an  expression
\begin{equation}
\label{eq:a-to-ag}
    a = \sum_{g \in \PPPP_n} c_g \cdot a_g \quad \quad (c_g \in \CC)
\end{equation}
of $a$ as a linear combination of the elements $a_g$ for partial permutations $g \in \PPPP_n$. Theorem~\ref{thm:forgotten} and Corollary~\ref{cor:partial-permutation-character} give a method for computing the irreducible character evaluation $\chi^\lambda(a)$ for $\lambda \vdash n$. In this final section, we carry out this method for a specific family of group algebra elements.

Let $I, J \subseteq [n]$ be subsets with $|I| = |J|$. Define the group algebra element $\{I,J\} \in \CC[\symm_n]$ by
\begin{equation}
    \{I,J\} = \sum_{\substack{w \in \symm_n \\ w(I) = J}} w \in \CC[\symm_n].
\end{equation}
That is, the element $\{I,J\}$ is the sum over all permutations which map the elements $I$ to those of $J$ in some order. In terms of partial permutations, we have
\begin{equation}
    \label{eq:set-to-ag}
    \{I,J\} = \sum_{g \in \PPPP_n(I,J)} a_g.
\end{equation}

We give a formula for $\chi^\lambda \{I,J\} := \chi^\lambda(\{I,J\})$ for any $\lambda \vdash n$ and any $I,J \subseteq [n]$ with $|I| = |J|$. This evaluation depends on the partition $\lambda \vdash n$ and the size of the intersection $I \cap J$. Our starting point is as follows.
A partial permutation $k$ is {\em cycle-free} if it does not contain any cycles. Let $\type(k)$ be the path type of a cycle-free partial permutation $k$.

\begin{lemma}
    \label{lem:cycle-free}
    Let $I,J \subseteq [n]$ satisfy $|I| = |J|$ and let $\chi: \symm_n \to \CC$ be a class function. We have
    \[
    \chi \{I,J\} = \sum_{T \subseteq I \cap J} |T|! \cdot  \sum_{\substack{k \in \PPPP_{[n] - T}(I-T,J-T) \\ k \text{ cycle-free}}}   \langle \ch(\chi), \vec{p}_{\type(k)} \cdot h_{|T|} \rangle
    \]
    where the inner sum is over all cycle-free partial permutations $k: (I-T) \xrightarrow{\, \sim \,} (J - T)$ of the set $[n]-T$. In particular, the index $\type(k)$ of the path power sum is a partition of $n -|T|$.
\end{lemma}

\begin{proof}
    Theorem~\ref{thm:forgotten}, Equation~\eqref{eq:pp-to-f}, and Equation~\eqref{eq:set-to-ag} imply
    \begin{equation}
        \chi\{I,J\} = \sum_{g \in \PPPP_n(I,J)} \langle \ch(\chi), \vec{p}_{\mu(g)} \cdot p_{\nu(g)} \rangle
    \end{equation}
    where the sum ranges over all partial permutations $g \in \PPPP_n(I,J)$, $\mu(g)$ is the path type of $g$, and $\nu(g)$ is the cycle type of $g$. For any $g \in \PPPP_n(I,J)$, there is a unique maximal subset $T_g \subseteq I \cap J$ such that $g$ restricts to a (full) permutation $w_g$ of $T_g$; we have $\nu(g) = \cyc(w_g)$. On the other hand, the partial permutation $g$ restricts to a cycle-free partial permutation $k_g \in \PPPP_{[n]-T_g}(I-T_g,J-T_g)$ of path type $\type(k_g) =\mu(g)$. We therefore have
    \begin{align}
        \sum_{g \in \PPPP_n(I,J)} \langle \ch(\chi), \vec{p}_{\mu(g)} \cdot p_{\nu(g)} \rangle &= 
        \sum_{T \subseteq I \cap J} \sum_{\substack{g \in \PPPP_n(I,J) \\ T_g = T}} \langle \ch(\chi), \vec{p}_{\type(k_g)} \cdot p_{\cyc(w_g)} \rangle \\
        &= \sum_{T \subseteq I \cap J} |T|! \cdot  \sum_{\substack{k \in \PPPP_{[n]-T}(I-T,J-T) \\ k \text{ cycle-free}}}   \langle \ch(\chi), \vec{p}_{\type(k)} \cdot h_{|T|} \rangle
    \end{align}
    where the second equality used the identity
    \begin{equation}
        \sum_{w \in \symm_T} p_{\cyc(w)} = |T|! \cdot h_{|T|}
    \end{equation}
    and the lemma is proved.
\end{proof}

There is significant cancellation on the right-hand side of Lemma~\ref{lem:cycle-free}. To describe this cancellation, we give the path power sum expansion of the symmetric function $n! \cdot h_n \cdot \vec{p}_\mu$.  For any composition $\mu$, we define
\begin{equation}
    m_\mu := m_{\sort(\mu)}, \quad 
    f_\mu := f_{\sort(\mu)}, \quad \text{and} \quad 
    \vec{p}_\mu := \vec{p}_{\sort(\mu)}.
\end{equation}
The cancellation in Lemma~\ref{lem:cycle-free} arises from the signs in the following lemma.

\begin{lemma}
    \label{lem:h-times-pp-to-pp}
    Let $n \geq 0$ and let $\mu = (\mu_1,\dots,\mu_r)$ be a composition with $r$ parts.  We have 
    \[
    n! \cdot h_n \cdot \vec{p}_{\mu} = 
    \sum_{j=0}^n (-1)^j \cdot  (n)_j \cdot   \sum_{\substack{I \subseteq [r] \\ |I| = j}}  \vec{p}_{(\mu_1 + \delta_{1 \in I}, \dots, \mu_r + \delta_{r \in I}, 1^{n-j})}.
    \]
    Here $(n)_j := n(n-1) \cdots (n-j+1)$ is a falling factorial and for $I \subseteq [r]$ and $1 \leq i \leq r$ we write
    \[
    \delta_{i \in I} := \begin{cases}
        1 & i \in I, \\
        0 & i \notin I.
    \end{cases} 
    \]
\end{lemma}


For example, suppose $n = 3$ and $\mu = (3,3,1,1)$. Lemma~\ref{lem:h-times-pp-to-pp} says that
\begin{multline*}
    3! \cdot h_3 \cdot \vec{p}_{3,3,1,1} = 
    (-1)^3 \cdot (3 \cdot 2 \cdot 1) \cdot ( \vec{p}_{4,4,2,1} + \vec{p}_{4,4,1,2} + \vec{p}_{4,3,2,2} + \vec{p}_{3,4,2,2} ) \\
    + (-1)^2 \cdot (3 \cdot 2) \cdot (\vec{p}_{4,4,1,1,1} + \vec{p}_{4,3,2,1,1} + \vec{p}_{4,3,1,2,1} + \vec{p}_{3,4,2,1,1} + \vec{p}_{3,4,1,2,1} + \vec{p}_{3,3,2,2,1} ) \\ + 
    (-1)^1 \cdot (3) \cdot ( \vec{p}_{4,3,1,1,1,1} + \vec{p}_{3,4,1,1,1,1} + \vec{p}_{3,3,2,1,1,1} + \vec{p}_{3,3,1,2,1,1}) + 
    (-1)^0 \cdot 1 \cdot \vec{p}_{3,3,1,1,1,1,1}.
\end{multline*}

\begin{proof}
    The product $e_n \cdot m_\mu$ of an elementary and monomial symmetric function has the monomial expansion
    \begin{equation}
        \label{eq:e-times-m-to-m}
        e_n \cdot m_\mu = \sum_{j=0}^n \sum_{\substack{I \subseteq [r] \\ |I| = j}} m_{(\mu_1 + \delta_{1 \in I}, \dots, \mu_r + \delta_{r \in I}, 1^{n-j})}.
    \end{equation}
    Applying $\omega$ to both sides of Equation~\eqref{eq:e-times-m-to-m} yields
    \begin{equation}
        \label{eq:h-times-f-to-f}
        h_n \cdot f_\mu = \sum_{j=0}^n \sum_{\substack{I \subseteq [r] \\ |I| = j}} f_{(\mu_1 + \delta_{1 \in I}, \dots, \mu_r + \delta_{r \in I}, 1^{n-j})}.
    \end{equation}
    By Equation~\eqref{eq:pp-to-f}, Equation~\eqref{eq:h-times-f-to-f}  is equivalent to the lemma.
\end{proof}

The right-hand side of Lemma~\ref{lem:cycle-free} may be simplified using the following result whose proof rests on Lemma~\ref{lem:h-times-pp-to-pp} and a sign-reversing involution.

\begin{lemma}
    \label{lem:pp-compact}
    Suppose $I,J \subseteq [n]$ are subsets with $|I| = |J|$. We have the path power sum expansion
    \begin{multline}
        \sum_{T \subseteq I \cap J} |T|!  \cdot \sum_{\substack{k \in \PPPP_{[n]-T}(I-T,J-T) \\ \text{$k$ cycle-free}}} \vec{p}_{\type(k)} \cdot h_{|T|}
        \\= \sum_{t=0}^{|I \cap J|} (-1)^t \cdot {|I \cap J| \choose t} \cdot (n- |I \cup J|)_t  \cdot \vec{p}_{2^{|I-J|+t}, 1^{n-2 \cdot (|I-J|+t)}}.
    \end{multline}
\end{lemma}

The left-hand side of Lemma~\ref{lem:pp-compact} involves path power sums $\vec{p}_{\type(k)}$ for arbitrary cycle-free partial permutations $k: (I-T) \to (J-T)$. The indexing partitions $\type(k)$ of these path power sums can have parts of any positive integer size $\leq n$,
but the right-hand side of Lemma~\ref{lem:pp-compact} involves only partitions with part sizes $\leq 2$.

\begin{proof}
    We use Lemma~\ref{lem:h-times-pp-to-pp} to give a combinatorial interpretation of the $\vec{p}$-expansion of the left-hand side. Given  $T \subseteq I \cap J$ and a cycle-free partial permutation $g \in \PPPP_n$, we say that $g$ is {\em $T$-valid} if every path of $g$ is of one of the following five forms:
    \begin{enumerate}
        \item isolated vertices $t$ for $t \in T$,
        \item isolated vertices $a$ for $a \in [n] - (I \cup J)$,
        \item paths of size 2 of the form $a \to t$ for some $a \in [n] - (I \cup J)$ and $t \in T$,
        \item paths of the form $i \to t \to i_1 \to \cdots \to i_s \to j$ where $i \in I - J$, $t \in T$, $i_1, \dots, i_s \in (I \cap J) - T$, and $j \in J - I$, and
        \item paths of the form $i \to i_1 \to \cdots \to i_s \to j$ where
        $i \in I - J$, $i_1, \dots, i_s \in (I \cap J) - T$, and $j \in J - I$.
    \end{enumerate}
    Observe that $T$-validity depends on the sets $I$ and $J$.
    The {\em $T$-sign} of a $T$-valid partial permutation $g$ is
    \begin{equation}
        \sign_T(g) := (-1)^{| \{t \in T \,:\, \text{$t$ occurs in a path of type (3) or (4)}\}|}.
    \end{equation}
    This sign depends on both $g$ and $T$.
    For example, suppose $I = \{1,2,3,4,5\}, J = \{3,4,5,6,7\}$, and $n = 9$. For $T = \{3,5\}$, a possible $T$-valid partial permutation is displayed as follows with the elements of $T$ shown in bold:
    $$
        1 \to {\bf 5} \to 4 \to 7 \quad \quad  2 \to 6 \quad \quad 8  \quad \quad 9 \to {\bf 3}.
    $$
    This partial permutation has $T$-sign $(-1)^2 = +1$. Another possible $T$-valid partial permutation is: 
    $$
    1 \to 4 \to 7 \quad \quad 2 \to {\bf 5} \to 6 \quad \quad {\bf 3} \quad \quad 8 \quad \quad 9
    $$
    This second partial permutation has $T$-sign $(-1)^1 = -1$. 
    Lemma~\ref{lem:h-times-pp-to-pp} allows us to rewrite the left-hand side of the current lemma as
    \begin{equation}
    \label{eq:pp-cancel-one}
        \sum_{T \subseteq I \cap J} |T|!  \cdot \sum_{\substack{k \in \PPPP_{[n]-T}(I-T,J-T) \\ \text{$k$ cycle-free}}} \vec{p}_{\type(k)} \cdot h_{|T|}
        \\= \sum_{T \subseteq I \cap J}   \sum_{\substack{g \in \PPPP_n \text{ is cycle-free} \\ \text{and $T$-valid}}} \sign_T(g) \cdot \vec{p}_{\type(g)}.
    \end{equation}
    Indeed, given a cycle-free partial permutation $k \in \PPPP_{[n]-T}(I-T,J-T)$, we complete $k$ to a $T$-valid cycle-free partial permutation  $g \in \PPPP_n$ by performing the following operations in all possible ways.
    \begin{itemize}
        \item We leave elements $t \in T$ as isolated vertices.
        \item We add elements $t \in T$ to isolated vertices $a \in [n] - (I \cup J)$ in $k$, yielding paths $a \to t$ of size 2.
        \item We splice elements $t \in T$ just after the start of paths of the form 
        \[
        i \to i_1 \to \cdots \to i_s \to j
        \]
        for $i \in I-J$, $j \in J-I$, and $i_1, \dots, i_s \in (I \cap J)-T$, yielding
        \[
        i \to t \to i_1 \to \cdots \to i_s \to j.
        \]
    \end{itemize}
    This precisely emulates the formula in Lemma~\ref{lem:h-times-pp-to-pp}.
    
    
    We perform a sign-reversing involution on the right-hand side of \eqref{eq:pp-cancel-one}. Let $\SSS$ be the set of pairs 
    \begin{equation}
        \SSS := 
        \left\{ (g,T) \,:\, \begin{array}{c}
            T \subseteq I \cap J \text{ and } \\
            g \in \PPPP_n \text{ is cycle-free and $T$-valid and} \\
            \text{has at least one path with more than $2$ vertices}
        \end{array}  \right\}.
    \end{equation}
    We show that the net contribution of $\SSS$ to the right-hand side of \eqref{eq:pp-cancel-one} vanishes. Let $(g,T) \in \SSS$ so that $g$ contains at least one path with more than $2$ vertices.  Let $i \to i_1 \to \cdots \to i_s \to j$ be the unique such path with $i \in I-J$ minimal. We have $s > 0$ and $i_1 \in I \cap J$. Define $T' \subseteq I \cap J$ by 
    \begin{equation}
        T' = \begin{cases}
            T \cup \{i_1\} & i_1 \notin T, \\
            T - \{i_1\} & i_1 \in T.
        \end{cases}
    \end{equation}
    Then $g$ is a $T'$-valid partial permutation with
    \begin{equation}
        \sign_{T'}(g) = -\sign_T(g).
    \end{equation}
    Furthermore, regarding $g$ as a $T'$-valid partial permutation we have $(T')' = T$ so that 
    \begin{equation}
        \label{eq:p-cancel}
        \sum_{(g,T) \in \SSS} \sign_T(g) \cdot \vec{p}_{\type(g)} = 0.
    \end{equation}
    Combining Equations~\eqref{eq:pp-cancel-one} and \eqref{eq:p-cancel} gives
    \begin{equation}
    \label{eq:pp-cancel-two}
        \sum_{T \subseteq I \cap J} |T|!  \cdot \sum_{\substack{k \in \PPPP_{[n]-T}(I-T,J-T) \\ \text{$k$ cycle-free}}} \vec{p}_{\type(k)} \cdot h_{|T|} = 
        \sum_{T \subseteq I \cap J}   \sum_{\substack{g \in \PPPP_n\text{ is cycle-free,} \\ \text{$T$-valid and} \\ \text{has no paths with more than $2$ vertices}}} \sign_T(g) \cdot \vec{p}_{\type(g)}.
    \end{equation}
    We claim that the right-hand side of Equation~\eqref{eq:pp-cancel-two} coincides with the right-hand side of the lemma. Indeed, a $T$-valid partial permutation $g \in \PPPP_n$ with no paths with more than $2$ vertices satisfies \[\sign_T(g) = (-1)^{|T|} \quad \text{and} \quad  \type(g) = (2^{|I-J| + |T|}, 1^{n - 2(|I-J| + |T|)}).\]
    Such a partial permutation can be formed by selecting $|T|$ elements of $[n] - (I \cup J)$ upon which to append an element of $T$ in 
    $(n - |I \cup J|)_{|T|}$ ways.
\end{proof}

Lemmas~\ref{lem:cycle-free} and \ref{lem:pp-compact} imply that for any $I,J \subseteq [n]$ with $|I| = |J|$ and any class function $\chi: \symm_n \to \CC$ we have
\begin{equation}
\label{eq:general-class}
    \chi\{I,J\} = \sum_{t=0}^{|I\cap J|} (-1)^t \cdot {|I \cap J| \choose t} \cdot (n-|I \cup J|)_t \cdot \langle \ch(\chi), \vec{p}_{2^{|I-J|+t}, 1^{n-2(|I-J|+t)}} \rangle.
\end{equation}
Thus, if $\chi = \chi^\lambda$ is an irreducible character of $\symm_n$, evaluating $\chi^\lambda \{I,J\}$ amounts to expanding path power sums of the form $\vec{p}_{2^a 1^b}$ in the Schur basis. Thanks to Theorem~\ref{thm:path-power-to-schur} this is an easy task.

\begin{lemma}
    \label{lem:two-row-pp-to-s}
    For $a,b \geq 0$ with $2a+b=n$, the path power sum $\vec{p}_{2^a1^b}$ has Schur expansion
    \[
    \vec{p}_{2^a 1^b} = a! b! \cdot \sum_{i=0}^a (-1)^i \cdot {a+b-i \choose b} \cdot s_{n-i,i}.
    \]
\end{lemma}

\begin{proof}
    With Theorem~\ref{thm:path-power-to-schur} and Equation~\eqref{eq:pp-to-f} as motivation, we consider the set $\MMM$ of monotonic ribbon tilings consisting of $a$ ribbons of size 2 and $b$ ribbons of size $1$. Any tiling $T \in \MMM$ is of shape $\lambda = (n-i,i)$ for some $0 \leq i \leq a$ and we have $\sign(T) = (-1)^i$. For fixed $i$, a tiling $T \in \MMM$ of shape $(n-i,i)$ will have $a-i$ horizontal ribbons of size 2 and $b$ ribbons of size 1 in the first row; the binomial coefficient ${a+b-i \choose b}$ counts how these ribbons can be ordered.
\end{proof}

We have all of the tools we need to present our formula for the irreducible character values $\chi^\lambda\{I,J\}.$ These character evaluations vanish unless $\lambda$ has $\leq 2$ parts.

\begin{theorem}
    \label{thm:set-irreducible-character}
    Let $I,J \subseteq [n]$ satisfy $|I| = |J|$ and let $\lambda \vdash n$. We have 
    \[
    \chi^\lambda \{I,J\} = 0 \quad \quad \text{if $\lambda$ has $>2$ parts.}
    \]
    If $\lambda = (n-i,i)$ has $\leq 2$ parts we have 
    \begin{equation*}
    \chi^\lambda \{I,J\} =\\ \sum_{t=\max(0,i-|I -J|)}^{|I \cap J|} (-1)^{t+i} \cdot {|I \cap J| \choose t} \cdot (n - |I \cup J|)_t \cdot b_t! \cdot (n-2  b_t)! \cdot {n- b_t - i \choose n - 2 b_t},
    \end{equation*}
    where $b_t := |I-J|+t.$
\end{theorem}

\begin{proof}
    Lemma~\ref{lem:cycle-free} applied to $\chi = \chi^\lambda$ implies that 
    \begin{align}
        \chi^\lambda \{I,J\} &= \text{coefficient of $s_\lambda$ in } 
        \sum_{T \subseteq I \cap J}  \sum_{\substack{k \in \PPPP_{[n]-T}(I-T,J-T) \\ k \text{ cycle-free}}} |T|! \cdot h_{|T|} \cdot \vec{p}_{\type(k)} \\
        &= \text{coefficient of $s_\lambda$ in } \sum_{t=0}^{|I \cap J|} (-1)^t \cdot {|I \cap J| \choose t} \cdot (n-|I\cup J|)_t \cdot \vec{p}_{2^{b_t},1^{n-2 \cdot b_t}},
    \end{align}
    where the second equality used Lemma~\ref{lem:pp-compact}. Lemma~\ref{lem:two-row-pp-to-s} implies that the above coefficient of $s_\lambda$ coincides with the formula in the proposition.
\end{proof}

\section*{Acknowledgments}

This project was initiated at the 2026 Spring Southeastern AMS meeting in Savannah. The authors thank John Shareshian and Michelle Wachs for organizing this meeting, bringing them together, and allowing BR to notice the similarity between tunnel hook coverings and monotonic ribbon tilings. The authors thank Ed Allen, Zach Hamaker, and Sarah Mason for many helpful conversations. BR was partially supported by NSF grant DMS-2246846.

\bibliographystyle{plain}
\bibliography{bibliography.bib}









\end{document}